\documentclass{amsart}
\usepackage{amsmath,amsthm,amssymb}
\usepackage{lipsum} 
\usepackage[utf8]{inputenc}
\usepackage{esint}
\usepackage{comment}
\usepackage[all]{xy}
\usepackage{amsfonts} 
\usepackage{xcolor}
\usepackage{tikz}
\usepackage[colorlinks=true]{hyperref}
\usepackage{bbold}
\usepackage{enumerate}
\usepackage{calc}
\usepackage{bbold}
\usepackage[colorinlistoftodos,prependcaption,textsize=tiny]{todonotes}
\usepackage{todonotes}
\usepackage[shortcuts]{extdash}
\usepackage[style=alphabetic,maxnames=200,maxalphanames=6,giveninits]{biblatex}
\makeatletter
\numberwithin{equation}{section}
\theoremstyle{plain}
        \newtheorem{theorem}{Theorem}[section]

        \newtheorem{lemma}[theorem]{Lemma}

        \newtheorem{remark}[theorem]{Remark}

\newtheorem*{theorem*}{Theorem}
\newtheorem*{definition*}{Definition}
\newtheorem*{proposition*}{Proposition}

\newcommand{\R}{\mathbb{R}}

\newcommand{\bg}{\mathbb{g}}

\newcommand{\bG}{\mathbb{G}}

\newcommand{\Id}{\mathsf{id}}

\newcommand{\wn}{{\widetilde\nabla}}

\newcommand{\bbnabla}{%
    \,\nabla\mkern-12mu\nabla_{i,j}%
}

\newcommand{\Do}{\R^{2d}}

\def\aS{\mathsf S}
\def\aZ{\mathsf Z}
\def\az{\mathsf z}

\def\aJ{\mathsf J}

\newcommand{\Tau}{\mathcal{T}}

\def\QB{Q^{\mathsf{B}}}
\def\QL{Q^{\mathsf{L}}}

\def\DB{\cD^{\mathsf{B}}}
\def\DL{\cD^{\mathsf{L}}}

\def\HL{\cH_{\mathsf{linear}}}

\newcommand{\cD}{\mathcal{D}}
\newcommand{\cH}{\mathcal{H}}

\newcommand{\cB}{\mathcal{B}}

\newcommand{\cP}{\mathbb{P}}

\newcommand{\aB}{\mathsf B}
\newcommand{\aL}{\mathsf L}
\newcommand{\aM}{\mathsf M}
\newcommand{\aaE}{\mathsf E}

\newcommand{\G}{\R^{3d}}
\newcommand\supp{\operatorname{Supp}}
\renewcommand{\d}{\partial} 
\newcommand{\defeq }{\mathop{=}\limits^{\textrm{def}}}
\newcommand{\dd}{{\,\rm d}}

\title{Multi-species kinetic models: GENERIC formulation and Fisher information}

\author[H.~Duong]{Manh Hong Duong}
\author[Z.~He]{Zihui He}
\address[H.~Duong]
{School of Mathematics, University of Birmingham, UK}
\email{h.duong@bham.ac.uk}
\address[Z.~He]
{Fakult\"at f\"ur Mathematik, Universit\"at Bielefeld, Postfach 100131, 33501 Bielefeld, Germany}
\email{zihui.he@uni-bielefeld.de}

\date{\today}

\begin{document}

\begin{abstract}
In this paper, we study the GENERIC structures of multi-species spatially inhomogeneous Boltzmann and Landau equation with Bose–Einstein, Maxwell–Boltzmann, and Fermi–Dirac statistics. In addition, under suitable assumptions on the collision kernels, we show that the Fisher information for the multi-species spatially homogeneous Boltzmann equation is non-increasing in time.
\end{abstract}

\maketitle
\section{Introduction}
In this paper we study the 
\textit{multi-species Boltzmann and Landau equations}. Our aim is twofold: firstly, we formulate these systems into the so-called \textit{GENERIC framework}, which is a well-established framework for non-equilibirum systems, and secondly we show that the Fisher information of the spatially homogeneous Boltzamnn equation is non-increasing in time. 
\subsection{Multi-species Boltzmann equations} We consider the multi-species Boltzmann equation describing the evolution of the densities of $N$ species of particles of different types representing Bose/ Maxwell/Fermi statistics:
\begin{equation}
\label{eq: sBE}
\partial_t f_i+v\cdot\nabla_x f_i=\QB_i(F,F),
\end{equation}
where $f_i=f_i(t,x,v)$ for  $i=1,\ldots, N$ denotes the density of the i-th species at time $t$, position $x\in\R^d$ and velocity $v\in\R^d$, and $F(t,x,v)=F=(f_1,\dots,f_N)^T$.

The left-hand side describes the advection of the density under free-transport of the particles. The collision term $\QB_i(F,F)$ on the right-hand side, which depends on all the species, is given by
\begin{equation}
\label{B:Q-ij} 
\begin{aligned}
\QB_{i}(F,F)&=\sum_{j=1}^N\QB_{ij}(F,F),\\
\QB_{ij}(F,F)&=\int_{\R^d\times S^{d-1}}B_{ij}
\big(f_i'f_{j}'\tau_i(f_i)\tau_j(f_{j})-f_if_j\tau_i(f_i')\tau_j(f_{j}')\big)\dd v_j\dd \omega, 
\end{aligned}
\end{equation}
where 
\begin{align*}
    \tau_i(f)=1+\alpha_i f.
\end{align*}
The parameter $\alpha_i$ encodes the type of particles:
\begin{equation}
\alpha_i=\begin{cases}
1\quad\text{Bose particles},\\
-1\quad\text{Fermi particles},\\
0\quad\text{Maxwell particles}.
\end{cases}
\end{equation}
We use the subscript $i$ to indicate that different species may be of different type of particles \cite{escobedo2003homogeneous}.

We can write the system \eqref{eq: sBE} in a vector form as follows
\begin{equation}
\label{IBE}
\begin{aligned}
\d_t F+v\cdot \nabla_x F=\QB(F,F),\quad (t,x,v)\in [0,T]\times \Do,
\end{aligned}
\end{equation}
where $\QB=(\QB_1,\ldots, \QB_N)^T$. The collision operator $Q_{ij}^B$ captures interactions between particles between $i$-th and $j$-th species, which is described below. 

Let $m_i$ denote the conserved mass of the i-th species.
Let $v_i$ and $v_j$ denote the pre-collision velocities of particles from species $i$ and $j$. The post-collision velocities $v_i'$ and $v_j'$ are given by 
\begin{align}
\label{post-p}   v_i'&=\frac{m_iv_i+m_jv_j+m_j|v_i-v_j|\omega}{m_i+m_j},\quad v_j'=\frac{m_iv_i+m_jv_j-m_i|v_i-v_j|\omega}{m_i+m_j}
\end{align}
for some $\omega\in S^{d-1}$. 
For a single collision, the following momentum and energy conservation laws hold
\begin{equation}
\label{B:cl}
m_iv_i+m_j v_j'=m_iv_i+m_j v_j,\quad m_i|v_i'|^2+m_j |v_j'|^2=m_i|v_i|^2+m_j |v_j|^2.
\end{equation}

We write  
\begin{align*}
f_i=f_i(x,v_i),\quad f'_{i}=f_i(x,v_i')\quad \text{and}\quad 
\tau_i=\tau_i(f_i),\quad 
\tau_i'=\tau_i(f_i').
\end{align*}

Let the collision kernel $B=(B_{ij})\in\R^{N\times N}$ be a symmetric matrix taking the following form
\begin{equation}
\label{B}
    B_{ij}(|v_i-v_j|,\omega)=\alpha_{ij}\big(|v_i-v_j|\big)b_{ij}(\theta_{ij}),\quad \theta_{ij}=\arccos\frac{(v_i-v_j\cdot \omega)}{|v_i-v_j|},
\end{equation}
where the functions $\alpha_{ij},\,b_{ij}:\R\to\R_+$ are symmetric, i.e. $\alpha_{ij}=\alpha_{ji}$ and $b_{ij}=b_{ji}$ . In the following, we do not distinguish the notations $b(\theta)$ and $b(\cos \theta)$. Notice that $ B_{ij}(|v_i-v_j|,\omega)= B_{ij}(|v_i'-v_j'|,\omega)$.
Notice that, for the single piece case ($N=1$), we recover the (quantum) Boltzmann equation.

We refer to \cite{briant2016boltzmann} and references therein for more information about the multi-species Boltzmann equation. 
\subsection{Multi-species Landau equations}
The multi-species Landau equation can be written as 
\begin{equation}
\label{ILE}
\begin{aligned}
\d_t f_i+v\cdot \nabla_x f_i=\QL_i(F,F),\quad (t,x,v)\in [0,T]\times \Do,
\end{aligned}
\end{equation}
where the Landau collision operator is given by  $Q_L=(\QL_{1},\dots,\QL_{N})^T$, where 
\begin{equation*}
\label{cla:Q-L}
\begin{aligned}
\QL_{i}(F,F)&=\sum_{j=1}^N\QL_{ij}(F,F),\\
\QL_{ij}(F,F)&=\nabla_{v_i}\cdot \int_{\R^d}A_{ij}\Pi_{ij}
\Big(\frac{1}{m_i}f_{j}\tau_j\nabla_{v_i} f_i-\frac{1}{m_j}f_i\tau_i(\nabla_{v_j} f_{j})\Big)\dd v_j.    
\end{aligned}
\end{equation*}
We write $\tau_i=\tau_i(f_i)$.
In the above, $\Pi=(\Pi_{ij})\in\R^{N\times N}$ denotes the projection matrix given by 
\begin{align*}
 \Pi_{ij}=\Id_d-\frac{(v_i-v_j)\otimes (v_i-v_j)}{|v_i-v_j|^2}.
\end{align*}

The collision kernel $A=(A_{ij})\in \R^{N\times N}$ satisfies $A_{ij}\ge 0$ and symmetric in the sense that
\begin{equation*}
m_iA^{ij}=m_jA^{ji}\quad\text{for all}\quad i,j=1,\dots,N.
\end{equation*}
The grazing limit of the multi-species Boltzmann equation \eqref{IBE} leads to the Landau equation \eqref{ILE}, in this case, the collision kernel $A$ is given by 
\begin{equation}
\label{A}
 A_{ij}=\frac{|v_i-v_j|^2\alpha_{ij}(|v_i-v_j|)}{m_i}.
\end{equation}
For completeness, we sketch the grazing limit in Appendix \ref{sec:GL} and refer to \cite{gualdani2017spectral} and references therein for more information about the multi-species Landau equation.

The first aim of the present paper is to show that the multi-species Boltzmann and Landau equations can be both formulated in the so-called GENERIC framework of non-equilibrium thermodynamics. We recall this framework next.  
\subsection{GENERIC framework}
In non-equilibrium thermodynamics, 
GENERIC (General Equation for Non-Equilibrium Reversible-Irreversible Coupling), which was first introduced in the context of complex fluids  \cite{GO97a,GO97b}, describes a class of thermodynamic evolution equations for  systems with both reversible and irreversible dynamics.

In a nutshell, a GENERIC equation with building blocks $\{\aL,\aM,\aS,\aaE\}$ describing the evolution of an unknown $\az$ in a state space $\mathsf Z$ is given by
\begin{equation}
\label{eq:generic}
\partial_t \az = \underbrace{\aL \dd\aaE}_{\text{reversible dynamics}}+\underbrace{\aM \dd\aS}_{\text{irreversible dynamics}}.
\end{equation}
Here, $\aL$ is an antisymmetric operator satisfying the Jacobi identity, while $\aM$ is a symmetric and positive semi-definite operator. The functionals $\aaE$ and $\aS$ denote the energy and entropy, respectively, and $\dd\aaE$, $\dd\aS$ are their appropriate differentials. The following degeneracy (non-interaction) conditions are satisfied
\begin{equation*}
\label{eq:degen}
\aL(\az) \dd \aS(\az)=0\quad\text{and}\quad \aM(\az) \dd\aaE(\az) = 0\quad\text{for all }\az.
\end{equation*}
These conditions ensure that, for any solution of \eqref{eq:generic}, the energy $\aaE$ is conserved while the entropy $\aS$ is non-decreasing. In other words, the first and second laws of thermodynamics are automatically satisfied for GENERIC systems.
The purely dissipative part of the dynamics is commonly referred to as a gradient flow. We refer to \cite{Ott05,pavelka2018multiscale} for further details on the GENERIC framework.
\subsection{Fisher information}
The Fisher information plays a crucial role in the study of kinetic equations \cite{Vil25}. In recent breakthroughs \cite{GS25,ISV26}, it is shown that the Fisher information is monotone decreasing in time along solutions, which imply the global existence of smooth solutions,  of the space-homogeneous Boltzmann and Landau equations. 

The Fisher information for the spatially homogeneous multi-species Boltzmann equation is given by
\begin{equation}
\label{Fisher-cI-F}
\mathcal{I}(F)\defeq\sum_{i=1}^NI(f_i),\quad I(f)\defeq\int_{\R^d}|\nabla \log f|^2 f\dd v.
\end{equation}
Motivated by and by employing the techniques of \cite{GS25,ISV26}, as the second aim of the paper, in Theorem \ref{thm:Fisher}, we will show the monotonicity of $\mathcal{I}$ along smooth solutions of spatially homogeneous multi-species Boltzmann equations.

\subsection{Related works}
The Boltzmann and Landau equations are fundamental models in kinetic theory. There has been a huge literature, both in physics and mathematics, on these equations physics. Multispecies Boltzmann and Landau equations, as well as related kinetic models, have studied in, for instance, \cite{jüngel2025crossdiffusionlimitsmultispecieskinetic,CHV24,BD16,daus2016hypocoercivity}. In particular, a two-species Boltzmann equation describing interactions between Bose–-Einstein and Fermi–-Dirac particles was studied in \cite{escobedo2003homogeneous}. 
We refer the reader to the monograph \cite{Cer88} and survey article \cite{villani2002review} for more information about the two equations. We now discuss recent works on the GENERIC structure and Fisher information of the \textit{single-species} Boltzmann and Landau equations (i.e., with $N=1$) that are directly related to this paper.

The GENERIC formulation for the spatial inhomogeneous Boltzmann equation has been shown formally in \cite{Ott97,grmela2018generic}, and it is rigorously studied in spatial homogeneous and delocalised cases in \cite{Erb23,EH25}. The GENERIC structure for the spatial homogeneous and delocalised cases  are studied in \cite{carrillo2024landau,DuongHe2025}. The grazing limit via variational structures are studied in \cite{carrillo2022boltzmann,DGH25}.

As already mentioned earlier in the previous section, in \cite{GS25,ISV26}, it is shown that the Fisher information for the single-species Landau and Boltzmann equations decays in time under appropriate collision kernels. These estimates play an important role in the study of well-posedness, asymptotic behaviour and stability. The decay of Fisher information of multi-species Landau equations have been established recently in \cite{JWY25,Zhu25}. 
\subsection{Organization of the paper}

 In Section \ref{sec:main}, we show the GENERIC building blocks for multi-species Boltzmann and Landau equations, and the grazing limit. In Section \ref{sec:Fisher}, we show that the Fisher information of the spatially homogeneous multi-species Boltzmann equation also decays in time under appropriate collision kernels. The grazing limit from the multi-species Boltzmann equation to the multi-species Landau equation is sketched in the appendix.

\subsection*{Acknowledgements}
M. H. D is funded by an EPSRC Standard Grant EP/Y008561/1. Z.~H. is funded by the Deutsche Forschungsgemeinschaft (DFG, German Research Foundation) – Project-ID 317210226 – SFB 1283.

\section{GENERIC structure of the multi-species  Boltzmann and Landau equations}\label{sec:main}
\subsection{GENERIC structure of the multi-species  Boltzmann equation}
In this section, we formulate the multi-species Boltzmann equation \eqref{eq: sBE} into the GENERIC framework. To this end, we need to generalize the definition of discrete gradient and divergence operators introduced in \cite{EH25} for the single-species system to the multi-species one.

For $\Phi=(\phi_1,\dots,\phi_N)^T$, $\phi_i=\phi(x,v_i)$, we define the Boltzmann gradient $\overline\nabla:\R^N\to \R^{N\times N}$ as follows
\begin{align*}
    \overline\nabla  \Phi=\big(\overline\nabla^{ij}\Phi\big)_{i,j=1}^N\quad \text{and}\quad \overline\nabla^{ij}\Phi=\phi_{j}'+\phi_i'-\phi_{j}-\phi_i. 
\end{align*}

For any $\Phi\in\R^N$ and $G=\big(g_{ij}\big)_{i,j=1}^N\in\R^{N\times N}$, $g_{ij}=g_{ij}(x,v_i,v_j,v_i',v_j')$, the following integration by parts formula holds
\begin{align}
\sum_{i,j=1}^N\int_{\G\times S^{d-1}}g_{ij}  \overline\nabla^{ij}\Phi\dd\omega\dd v_i\dd v_j\dd x=-\sum_{i=1}^N\int_{\Do}\overline\nabla^i\cdot G \phi_i\dd v_i\dd x. \label{B-M:IP}
\end{align}
Then $\overline \nabla\cdot :\R^{N\times N}\to\R^N$ is given by
\begin{gather*}
\overline \nabla \cdot G=\big(\overline \nabla^1 \cdot G,\dots,\overline \nabla^N \cdot G\big)^T,\\
\overline \nabla^i\cdot G=\sum_{j=1}^N\int_{\R^d\times S^{d-1}} g_{ij}+g^*_{ij}-g_{ij}'-(g_{ij}')^*\dd \omega\dd v_j,
\end{gather*}
where we write 
\begin{gather*}
 g_{ij}=g_{ij}(x,v_i,v_j,v_i',v_j'),\quad g^*_{ij}=g_{ij}(x,v_j,v_i,v_j',v_i'),\\
 g_{ij}'=g_{ij}(x,v_i',v_j',v_i,v_j),\quad
 (g_{ij}')^*=g_{ij}(x,v_j',v_i',v_j,v_i).
\end{gather*}
Let $\Lambda(s,t)=\frac{s-t}{\log s-\log t}$, $s,t>0$, denote the logarithm mean. We define the matrix
\begin{align*}
    \Lambda(F)=\big(\Lambda_{ij}(F)\big)_{i,j=1}^N\in \R^{N\times N}\quad\text{and}\quad \Lambda_{ij}(F)=\Lambda\big(\tau_j'\tau_i',\tau_j \tau_i\big).
\end{align*}

For $i$th-spice, the entropy function is given by 
\begin{equation*}
\label{h_i}
 h_i(f)=f\log f-\tau_i(f)\log\tau_i(f)-f\mathbb{1}_{\alpha_i=0}.
\end{equation*}
In the Fermi case $\alpha_i=-1$, we take $h_i(f)=+\infty$ if $f\notin[0,1]$. Let $\cH(f)$ denote the entropy functional
\begin{equation*}
\label{H:i}
\cH_i(f)=\int_{\Do}h_i(f)\dd x\dd v
\end{equation*}
in the case that $\max\big(h_i'(f),0\big)$ is integrable. Otherwise, we take $\cH_i(f)=+\infty$.
For the vector-valued function $F$, we define the entropy functional
\begin{equation}
\label{H}
\cH(F)=\sum_{i=1}^N\cH_i(f_i).
\end{equation}
Let $\dd\cH(F)=\big(\dd\cH_1(f_1),\dots,\dd\cH_N(f_N)\big)^T$ denote the functional differentiation, and 
\begin{equation}
\dd\cH_i(f)=h_i'(f)=\log \frac{f}{\tau_i(f)}.
\end{equation}

The Boltzmann collision operator \eqref{B:Q-ij} $Q^{ij}_B(F,F)$ can be written as 
\begin{align*}
\QB_{ij}(F,F)&=\int_{\R^d\times S^{d-1}}B_{ij}\Lambda^{ij}(F) \overline\nabla^{ij}\dd\cH(F)\dd \omega\dd v_*.
\end{align*}
For matrices $A=\big(a^{ij}\big)$ and $ B=\big(b_{ij}\big)\in \R^{N\times N}$, we use the notation of the Hadamard product
\begin{align*}
A\circ B=\big(a_{ij}b_{ij}\big)_{i,j=1}^N\in \R^{N\times N}.
\end{align*}
Then the Boltzmann equation \eqref{IBE} can be written as
\begin{align*}
\d_tF+v\cdot\nabla_xF=\frac14\overline\nabla\cdot\big(B\circ \Lambda\circ \overline\nabla\dd\cH( F)\big).
\end{align*}

By using of the integration by parts formula \eqref{B-M:IP}, the following weak formulation holds for the Boltzmann equation \eqref{IBE}
{\footnotesize\begin{equation}
    \label{weak-B}
    \begin{aligned}
        \int_{\Do}\Phi_0\cdot F_0-\int_{0}^T\int_{\Do}F\cdot (\d_t+v\cdot\nabla_x)\Phi =-\frac14\int_{0}^T\int_{\G\times S^{d-1}}\overline\nabla\Phi:\big(B\circ \Lambda\circ \overline\nabla \dd\cH( F)\big)
    \end{aligned}
\end{equation}}
for all test functions $\Phi\in\R^N$. For matrices $A=\big(a_{ij}\big)$ and $B=\big(b_{ij}\big)\in \R^{N\times N}$, we define $A:B:=\sum_{i=1}^Na_{ij}b_{ij}$. 

One can also check \eqref{weak-B} straightforwardly by adding up the weak $\QB_{ij}$ and $\QB_{ji}$ terms
\begin{equation}
\label{check}
\begin{aligned}
\int_{\Do}\phi^i\QB_{ij}(F,F)\dd x\dd v_i&=\int_{\G\times S^{d-1}}B_{ij}(\phi_i'-\phi_i)f_if_j\tau_i'\tau_j'\dd\omega\dd v_i\dd v_j\dd x,\\
\int_{\Do}\phi^j\QB_{ji}(F,F)\dd x\dd v_j&=\int_{\G\times S^{d-1}}B_{ji}\big(\phi_{j}'-\phi_{j}\big)f_jf_i\tau_j'\tau_i'\dd\omega\dd v_i\dd v_j\dd x
\end{aligned}
\end{equation}
and by using the symmetrical $B_{ij}=B_{ji}$, we have 
\begin{align*}
&\int_{\Do}\phi^i\QB_{ij}(F,F)\dd v_i\dd x+\int_{\Do}\phi^j\QB_{ji}(F,F)\dd x\dd v_j\\
=&{}\int_{\G\times S^{d-1}}B_{ij}\big(\phi_{j}'+\phi_i'-\phi_{j}-\phi_i\big)f_if_j\tau_i'\tau_j'\dd\omega\dd v_i\dd v_j\dd x\\
=&{}-\frac12\int_{\G\times S^{d-1}}B_{ij}\big(\phi_{j}'+\phi_i'-\phi_{j}-\phi_i\big)\big(f_i'f_j'\tau_i\tau_j-f_if_j\tau_i'\tau_j'\big)\dd\omega\dd v_i\dd v_j\dd x\\
=&{}-\frac12\int_{\G\times S^{d-1}}\overline\nabla^{ij}\Phi B_{ij} \Lambda^{ij}(F)\overline\nabla^{ij}\dd\cH(F)\dd\omega\dd v_i\dd v_j\dd x.
\end{align*}

We define the momentum $M$ and energy $E$ for the system
\begin{align*}
   M(F)=\sum_{i=1}^N\int_{\Do}m_iv_if_i\dd x\dd v_i\quad\text{and}\quad E(F)=\sum_{i=1}^N\int_{\Do}\frac{m_i}{2}|v_i|^2f_i\dd x\dd v_i.
\end{align*}
We have the following the mass conservation law for single species
\begin{equation}
\label{mass-consv}
\int_{\Do}(f_i)_t\dd x\dd v_i =\int_{\Do}(f_i)_0\dd x\dd v_i\quad\forall t\in[0,T]\quad i=1,\dots,N.   
\end{equation}
By \eqref{B:cl}, the following momentum and energy conservation laws holds
\begin{equation}
    \label{B-me}
\begin{aligned}
M_t=M_0\quad\text{and}\quad E_t=E_0\quad\forall t\in[0,T].
\end{aligned}
\end{equation}

For the Boltzmann equation \eqref{IBE}, the entropy is non-increasing in time. The following entropy identity holds at least formally
\begin{equation}
\label{H-theorem}
\cH(F_t)-\cH(F_0)=-\int_0^t  \cD(F_s)\dd s \quad\forall t\in[0,T],
\end{equation}
where the entropy dissipation is given by
\begin{equation*}
  \DB(F)=\frac14\sum_{i,j=1}^N\int_{\G\times S^{d-1}} B_{ij}\Lambda^{ij} |\overline\nabla^{ij}\dd\cH(F)|^2 \dd\omega\dd v_i\dd v_j\dd x\ge 0.
\end{equation*}

We consider the space $\aZ$ to be the space that consists of all vector-valued Schwartz functions $F=(f_1,\dots,f_N)^T$ and $f_i\in\cP(\R^{2d})$ has probability density. The space $\aZ$ is endowed with the $L^2$-inner product. 
The multi-species Boltzmann equation \eqref{IBE} can be cast as GENERIC structure with the building block $\{\aL,\aM,\aaE,\aS\}$ with  
\begin{equation}
    \label{B-1:EL}
\begin{aligned}
&\aaE(F)=E(F)\quad\text{and}\quad \aS(F)=-\cH(F),
\end{aligned}
\end{equation}
and the operators $\aL$ and $\aM$ at $F\in\aZ$ are given by
\begin{equation}
    \label{B-1:ML}
\begin{aligned}
\aM(F)G&=-\frac{1}{4}\overline\nabla\cdot\Big(B\circ \Lambda(F)\circ\overline\nabla \xi\Big)\quad\text{and}\\
\aL(F)G&=-\nabla\cdot\Big(\frac{F}{m} \big(\aJ\nabla \xi\big)^T\Big),\quad \aJ=\begin{pmatrix}
    0& \mathsf{id}_d\\
    -\mathsf{id}_d&0
\end{pmatrix}
\end{aligned}
\end{equation}
for all smooth decaying functions $\xi=(\xi_1,\dots,\xi_N)^T$. In the above, we use the notation $\nabla =(\nabla_x,\nabla_v)^T$, and $\frac{F}{m}$ denotes the matrix
\begin{equation*}
\frac{F}{m}=\big({f_i}/{m_j}\big)_{i,j=1}^N\Id_N\in \R^{N\times N}. 
\end{equation*}
One can check straightforwardly that ${\aL,\aM,\aaE,\aS}$ form a GENERIC building block for \eqref{IBE}. The computations are very close to those in the single-species case; we refer the reader to \cite[Appendix]{EH25} for details.

The Boltzmann equation also admits non-quadratic, so-called $\cosh$-GENERIC building blocks, which are related to large-deviation structures, which can also be generalised to the multi-species case \eqref{IBE} analogously to \eqref{B-1:EL} and \eqref{B-1:ML}. We refer the reader to \cite{DGH25} for further discussion on the $\cosh$ entropy dissipation potentials. 

\subsection{GENERIC structure of the multi-species Landau equation}
In this section, we show the GENERIC structure of the multi-species Landau system \eqref{ILE}. Similarly as in the Boltzmann system, we need to extend the concept of discrete gradient and divergence operators from the single-species introduced in \cite{DuongHe2025} to the multi-species case.

For smooth decaying $\Phi=(\phi_1,\dots,\phi_N)^T$, $\phi_i=\phi_i(x,v_i)$, we define the Landau gradient
\begin{equation*}
    \widetilde\nabla  \Phi=\big(\widetilde\nabla^{ij}\Phi\big)_{i,j=1}^N\quad\text{and}\quad \widetilde\nabla^{ij}\Phi=\Pi_{ij}\Big(\frac{\nabla_{v_i} \phi_i}{m_i}-\frac{\nabla_{v_j}\phi_j}{m_j}\Big)\in \R^d. 
\end{equation*}
For any smooth decaying functions $\Phi$ and $\bG=\big(\bg_{ij}\big)_{i,j=1}^N$ with $\bg_{ij}=\bg_{ij}(x,v_i,v_j)$ and $\bg_{ij}\in\R^d$, we have the following integration by parts formula
\begin{align*}
\sum_{i,j=1}^N\int_{\G}\bg_{ij} \cdot \widetilde\nabla^{ij}\Phi\dd v_i\dd v_j\dd x=-\sum_{i=1}^N\int_{\Do}\phi_i\widetilde\nabla^i\cdot \bG   \dd v_i\dd x.\label{L-M:IP}
\end{align*}
 Then $\widetilde \nabla \cdot $ is given by
\begin{gather*}
\widetilde \nabla \cdot \bG=(\widetilde \nabla^1 \cdot \bG,\dots,\widetilde \nabla^N \cdot \bG)^T\in \R^N,\\
\widetilde \nabla^i\cdot \bG=\sum_{j=1}^N\frac{\nabla_{v_i}}{m_i}\cdot \int_{\R^d}\Pi_{ij}\big(\bg_{ij}-\bg_{ij}^*\big)\dd v_j.
\end{gather*}

We recall the definition of entropy \eqref{H}, and we have 
\begin{align*}
\widetilde\nabla^{ij}\dd\cH( F)&=\Pi_{ij}\Big(\frac{\nabla_{v_i} \dd\cH_i(f_i)}{m_i}-\frac{\nabla_{v_j} \dd\cH_j(f_j)}{m_j}\Big)\\
&=\Pi_{ij}\Big(\frac{\nabla_{v_i} f_i}{m_if_i\tau_i}-\frac{\nabla_{v_j} f_j}{m_jf_j\tau_j}\Big). 
\end{align*}

Let $M:=(m_1,\dots,m_N)\Id\in \R^{N\times N}$. Notice that $MA=\big(m_iA_{ij}\big)\in \R^{N\times N}$ is a symmetric matrix. Let $F\Tau:=\big(f_i\tau_i,\dots,f_N\tau_N\big)^T\in \R^N$.
By definition, the Landau collision operator can be written as
\begin{equation*}
 \label{div-Landau}   
\begin{aligned}
    \QL_{ij}(F,F)&=\frac{\nabla_{v_i}}{m_i}\cdot \int_{\R^d}m_iA^{ij}f_i\tau_if_{j}\tau_{j}\widetilde\nabla^{ij}\dd\cH(F)\dd v_j,\\
\QL(F,F)&=\frac12\widetilde\nabla\cdot\Big(M A\circ F\Tau\otimes F\Tau\circ \widetilde\nabla \dd\cH( F)\Big).
\end{aligned}
\end{equation*}

The following weak formulation holds for test functions $\Phi\in\R^N$
\begin{equation*}
    \label{weak-L}
    \begin{aligned}
     &\big\langle \QL(F,F),\Phi\big\rangle\\
     =&{}-\frac12\int_{\G}MA\circ F \Tau \otimes F\Tau \circ \widetilde\nabla \dd\cH( F):\widetilde\nabla\Phi\\
=&{}-\frac12\sum_{i,j=1}^N\int_{\G}m_iA_{ij}f_i\tau_if_j\tau_j \Pi_{ij}\Big(\frac{\nabla_{v_i} f_i}{m_if_i\tau_i}-\frac{\nabla_{v_j} f_j}{m_jf_j\tau_j}\Big)\cdot \Big(\frac{\nabla_{v_i} \phi_i}{m_i}-\frac{\nabla_{v_j} \phi_j}{m_j}\Big).
    \end{aligned}
\end{equation*}
Similar to the Boltzmann case \eqref{check}, one can check the above weak formulation by adding up the $\QL_{ij}$ and $\QL_{ji}$ terms.

Similar to the Boltzmann case, the mass, momentum and energy conservation laws \eqref{mass-consv} and \eqref{B-me} hold at least formally.
The entropy identity \eqref{H-theorem} holds with the Landau entropy dissipation given by
\begin{equation*}
  \DL(F)=\frac12\sum_{i,j=1}^N\int_{\G} m_i A_{ij}(F\Tau\otimes F\Tau)_{ij} \Big|\widetilde\nabla^{ij} \dd\cH(F)\Big|^2 \dd \eta\ge 0.
\end{equation*}

The multi-species Landau equation \eqref{ILE} has a GENERIC building block $\{\aL,\aM,\aaE,\aS\}$ with $\aL,\aaE,\aS$ are given by \eqref{B-1:EL} and \eqref{B-1:ML}, and the operator $\aM$ is replaced by 
\begin{equation*}
    \label{L-1:ML}
\begin{aligned}
\aM^{\aL}(F)\bG=-\frac{1}{2}\widetilde\nabla\cdot\Big(MA\circ F\Tau\otimes F\Tau \circ\widetilde\nabla \bG\Big).
\end{aligned}
\end{equation*}

\begin{remark}[Linearised equations]
We consider the following multi-species linearised Boltzmann and Landau equation
\begin{equation}
    \label{eq:linear}
\begin{gathered}
\d_t F+v\cdot\nabla_x F= Q_{\sf linear}(F)=Q(F,\mathsf{1})+Q(\mathsf{1},F)\\
 Q_{\sf linear}^{\aB}(F)=\frac{1}{4}\overline\nabla\cdot\big(B\circ\overline\nabla F\big),\quad Q_{\sf linear}^{\aL}(F)=\frac12\widetilde\nabla\cdot\Big(M A\circ  \widetilde\nabla F\Big),
\end{gathered}
\end{equation}
where $\mathsf{1}:=(1,\dots,1)^T\in \R^N$. The mass, momentum and energy conservation laws \eqref{mass-consv} and \eqref{B-me} hold at least formally.
We define the linear entropy
\begin{equation*}
\HL(F)=\frac12\sum_{i=1}^N\int_{\Do}f_i^2\dd v_i\dd x.
\end{equation*}
The following entropy identity holds at least formally
\begin{gather*}
\label{L-H}
\HL(F_t)-\HL(F_0)=-\int_0^t  \cD(F_s)\dd s \quad\forall t\in[0,T],
\end{gather*}
where the entropy dissipation in Boltzmann and Landau cases are given by
\begin{gather*}
\cD^{\aB}_{\sf linear}(F)=\frac14\sum_{i,j=1}^N\int_{\G\times S^{d-1}} B_{ij}|\overline\nabla^{ij} F|^2,\quad \cD^{\aL}_{\sf linear}(F)=\frac12\sum_{i,j=1}^N\int_{\G} m_i A_{ij} \Big|\widetilde\nabla^{ij}  F \Big|^2.
\end{gather*}

The linearised multi-species systems \eqref{eq:linear} have the GENERIC building blocks $\{\aL,\aM,\aaE,\aS\}$, where $\aaE$ and $\aL$ are given by \eqref{B-1:EL} and \eqref{B-1:ML}, the entropy $\aS(F)=-\HL(F)$, and $\aM$ in the Boltzmann and Landau cases are given by
\begin{equation*}
    \label{B-2:EL}
\begin{aligned}
\aM^{\aB}_{\sf linear}(F)G&=-\frac{1}{4 }\overline\nabla\cdot\Big(B\circ  \overline\nabla G\Big),\quad \aM^{\aL}_{\sf linear}(F)\bG=-\frac{1}{2}\widetilde\nabla\cdot\Big(MA\circ \widetilde\nabla \bG\Big).
\end{aligned}
\end{equation*}    
\end{remark}

\section{Fisher information of the spatially homogeneous multi-species Boltzmann equation}
\label{sec:Fisher}

The multi-species Boltzmann and Landau equations with vanishing spatial dependence on $x\in\R^d$ reduce to homogeneous models associated to gradient flow structures.
In this section, we generalise the Fisher information estimates for the single-species homogeneous Boltzmann equation in \cite{ISV26} to the multi-species cases \eqref{IBE} with taking $\tau_i(f_i)=1$ for all $i=1,\dots,N$
\begin{equation}
\label{B-homo}
\begin{aligned}
&\d_t f_i=\sum_{j=1}^NQ_{ij}(F,F),\quad f_i=f_i(t,v_i),\\
&Q_{ij}(F,F)=\int_{\R^d\times S^{d-1}}B_{ij}
\big(f_i'f_{j}'-f_if_{j}\big)\dd \omega \dd v_j .
\end{aligned}
\end{equation}
We recall the collision kernel $B_{ij}$ given in \eqref{B}
\begin{align}
\label{B-ij}
    B_{ij}(|v_i-v_j|,\omega)=\alpha_{ij}\big(|v_i-v_j|\big)b_{ij}(\theta_{ij}),\quad \theta_{ij}=\arccos\frac{(v_i-v_j\cdot \omega)}{|v_i-v_j|}.
\end{align}
In the following, we do not distinguish the notations $b(\theta)$ and $b(\cos \theta)$.
We recall the definition of the Fisher information for the spatially inhomogeneous system \eqref{B-homo} given in \eqref{Fisher-cI-F}
\begin{equation*}
\mathcal{I}(F)\defeq\sum_{i=1}^NI(f_i),\quad I(f)\defeq\int_{\R^d}|\nabla \log f|^2 f\dd v.
\end{equation*}
For a function $f:S^{d-1}\to \R_+$ such that $f(\omega)=f(-\omega)$, we define the following integro-differntial operator $\cB$
\begin{equation}
\label{def:cB}
    \cB f(\omega)=\int_{S^{d-1}}\big(f(\omega')-f(\omega)\big) b(\omega\cdot \omega')\dd\omega'.
\end{equation}
Let $\Gamma_\Delta$ denote the carr\'e du champ operator associated to Laplace-Beltrami operator on sphere $S^{d-1}$ given by 
\begin{align*}
\Gamma_\Delta(f,g):= \frac12\big(\Delta(fg)-\Delta f g- f\Delta g\big).    
\end{align*}
The carr\'e du champ operator associated to $\Delta $ and $\cB$  is defined as \begin{align*}
\Gamma^2_{\cB,\Delta}(f,g):=\frac12\big(\cB\Gamma_\Delta(f,g)-\Gamma_\Delta(\cB f,g)-\Gamma_\Delta(f,\cB g)\big).   
\end{align*} 

The following integro-differential functional inequality was proved in \cite{ISV26}
\begin{equation}
\label{log-sobolev}
    \int_{S^{d-1}}\Gamma^2_{\cB,\Delta}(\log f,\log f)f\dd\omega\ge \Lambda_b\int_{S^{d-1}\times S^{d-1}}\frac{|f(\omega)-f(\omega')|^2}{f(\omega)+f(\omega')}b_{ij}(\omega\cdot \omega')\dd \omega'\dd\omega,
\end{equation}
where the constant  $\Lambda_b>0$ denotes the optimal constant and  depends on the angle function $b$. 
The inequality \eqref{log-sobolev} is proved in \cite{ISV26} for the even case, that is, when $f(\omega)=f(-\omega)$. The proof relies on classical log-Sobolev and Poincaré inequalities on sphere to obtain an integro-differential formulation, where the evenness assumption only plays a role in optimal constant. Notice that in the multispecies setting, the evenness assumption only refer to the  interaction of the particles in the same species ($i=j$). 
We refer the reader to \cite{ISV26, Vil25} for a detailed discussion of the optimal constant in the even case. 

The following theorem on the decay of the Fisher information for the multi-species system \eqref{B-homo} is the main result of this section.
\begin{theorem}\label{thm:Fisher}
Let $\alpha_{ij}$ and $b_{ij}:\R\to\R_+$ satisfy $\alpha_{ij}=\alpha_{ji}$ and $b_{ij}=b_{ji}$ for all $i,j=1,\dots,N$, and assume that
\begin{equation}
\label{ass}
r\frac{|(\alpha_{ij})'(r)|}{\alpha_{ij}(r)}\le 2\sqrt{\Lambda_{ b_{ij}}}\quad \forall\, r>0.    
\end{equation}
Then the Fisher information of the multi-species Boltzmann equation \eqref{B-homo}–\eqref{B-ij} is non-increasing in time. 
\end{theorem}
\begin{proof}

Following the multi-species Landau case in \cite{JWY25}, we will first introduce a lifted system using the technique of doubling the variables, and then apply the Fisher information estimates from \cite{ISV26} to handle the two-species interactions. 

For a function $\widetilde F:\R^{2dN}\to\R_+$, $\widetilde{F}=
\widetilde{F}(\overline{v},\overline{v}_*)$ where $\overline v=(v_{1},\dots  v_{N})$ and $ \overline v_*=(v_{1*},\dots, v_{N*})\in \R^{dN}$, we define the Fisher information
\begin{align*}
    \widetilde I(\widetilde F)=\sum_{i=1}^{N}\frac{1}{m_i}\int_{\R^{2Nd}}|\nabla_{i}\log \widetilde F|^2\widetilde F \dd \overline v\dd \overline v_*,\quad \nabla_i=(\nabla_{v_i},\nabla_{v_{i*}})^T.
\end{align*}
We also define the $i$-marginals as follows
\begin{equation*}
    \pi_i\widetilde F=\int_{\R^{(2n-1)d}}\widetilde F\dd\hat v_i\dd v_*,
\end{equation*}
 where $\hat v_i=(v_1,\dots, v_{i-1},v_{i+1},\dots, v_{N})\in \R^{d(N-1)}$.  One can define the $i_*$-marginals analogously. Note that $\pi_i \widetilde F = \pi_{i_*} \widetilde F$ if $\widetilde F$ is symmetric in $v_i$ and $v_{i_*}$.

Under the symmetry assumption, similarly to the single-species case, the Fisher information is controlled by the sum of the Fisher informations of the marginals; see, for example, \cite[Lemma 5.1]{JWY25}
\begin{align*}
    \tilde I(\tilde F)\ge 2\sum_{i=1}^N\frac{1}{m_i}I(\pi_i\widetilde F).
\end{align*}

We define the spherical integro-differential operator $\widetilde
Q_{ij_*}(\widetilde F,\widetilde F)$  as follows 
\begin{align*}
\widetilde
Q_{ij_*}(\widetilde F,\widetilde F)&= \int_{S^{d-1}}B_{ij}(|v_i-v_{j_*}|,\omega')\big(\widetilde F_{ij_*}'-\widetilde F\big)\dd \omega', 
\end{align*}
where we define $\widetilde F'_{ij_*}$ as replacing $v_i$ and $v_{j_*}$ by post-collision velocities $v_i'$ and $v_{j_*}'$ given by \eqref{post-p} associated to $\omega'\in S^{d-1}$ as follows
\begin{gather*}
\widetilde F^{ij_*}:= \widetilde F(v_1,\dots, v_i',\dots, v_{j*}',\dots v_{N*}),\\
v_i'=\frac{m_iv_i+m_{j*}v_j+m_j|v_i-v_{j*}|\omega'}{m_i+m_j},\quad v_{j*}'=\frac{m_iv_i+m_jv_{j*}-m_i|v_i-v_{j*}|\omega'}{m_i+m_j}.
\end{gather*}

In the following, we write $f_{i*}=f_i(v_{i*})$. Without loss of generality, we take  $f_i$ and $f_{i*}\in\cP(\Do)$ for all $i=1,\dots,N$. 
We take $\widetilde F$ as the following form of tensorised function
\begin{align*}
      \widetilde F=\Pi_{i=1,\dots,N}^{\otimes}f_i\otimes f_{i*}.
\end{align*}
One can check straightforwardly that, for all $i,j=1,\dots,N$
\begin{align}
\label{goal-0}
\big\langle \widetilde I_{ij}'(f_i\otimes f_{j_*}),\widetilde
Q_{ij_*}(f_i\otimes f_{j_*})\big\rangle= \langle  I'(f_i),  Q_{ij}(F,F)\rangle+\langle I'(f_j), Q_{ji}(F,F)\rangle,
\end{align}
where $\langle I'(f),\xi\rangle$  denote the Gateaux derivative of $I$ in the direction of $\xi$ (similar definition for $\tilde{I}'$). 
We define the following integro-differential operator and Fisher information for tensor functions $f_i\otimes f_{j*}$
\begin{align*}
 &\widetilde
Q_{ij_*}(f_i\otimes f_{j*})=\int_{S^{d-1}}  B_{ij}\big((f_i'\otimes f_{j*}')-(f_i\otimes f_{j*})\big)\dd\omega',\\
 &\widetilde I_{ij}(f_i\otimes f_{j*})=\int_{\R^{2d}}|\bbnabla\log (f_i\otimes f_{j*})|^2(f_i\otimes f_{j*})\dd v_i\dd v_{j*},
\end{align*}
where ${\bbnabla}:=\big({\nabla_{v_i}}/{\sqrt{m_i}},{\nabla_{v_{j*}}}/{\sqrt{m_j}}\big)$. 
On the one hand, by using \cite[Lemma 5.2]{JWY25}, we have 
\begin{align}
\label{goal-1}
    \big\langle \widetilde I'(\widetilde F),\widetilde
Q_{ij_*}(\widetilde F,\widetilde F)\big\rangle
    =\big\langle \widetilde I_{ij}'(f_i\otimes f_{j_*}),\widetilde
Q_{ij_*}(f_i\otimes f_{j_*})\big\rangle,
\end{align}
where the proof does not involve the precise form of operator $Q$. 
On the other hand, the symmetrical of $\alpha$ and $b$ ensures that the right-hand-sides of \eqref{goal-0} and \eqref{goal-1} are identical.
Hence, to show Theorem \ref{thm:Fisher}, we only need to show
\begin{align}
\label{goal-2}
    \big\langle \widetilde I_{ij}'(f_i\otimes f_{j_*}),\widetilde Q_{ij_*}(f_i\otimes f_{j_*})\big\rangle\le 0,\quad \forall \,i,j=1,\dots,N.
\end{align}
To this end, we fix a pair of $i,j_*$.
We use the centre of mass coordinate to define
\begin{align*}
z=\zeta v_i+(1-\zeta )v_{j*},\quad \zeta=\frac{m_i}{m_i+m_j}.
\end{align*}
We write $v_i-v_{j*}=r\omega$ in the spherical coordinate with $r\in \R_+$ and $\omega\in S^{d-1}$. Notice that the pre- and post collision velocities can be written as 
\begin{align*}
v_i=z+(1-\zeta)r\omega,\quad     v_{j_*}=z-\zeta r\omega,\quad v_i'=z+(1-\zeta)r\omega',\quad     v_{j_*}=z-\zeta r\omega'. 
\end{align*}
We define
\begin{align*}
    g=g(z,r,\omega)=f_i\otimes f_{j*}\quad\text{and}\quad g'=g(z,r,\omega')=f_i'\otimes f_{j*}'.
\end{align*}
Notice that $g(\omega)=g(-\omega)$ only in the case of $i=j$.

Let $\cB_{ij}$ be defined as in \eqref{def:cB} associated to the angle function $b_{ij}$. Notice that
\begin{align*}
\widetilde Q_{ij}(g,g)=\alpha_{ij}(r)\int_{S^{d-1}}b_{ij} (\omega\cdot\omega')(g'-g)\dd\omega'=\alpha_{ij}(r)\cB_{ij} g.
\end{align*}

By using spherical coordinate, we have
\begin{align*}
    \nabla_{v_i}=\zeta \nabla_z+\omega\nabla_r+r^{-1}\nabla_\omega\quad \text{and}\quad  \nabla_{v_{j_*}}=(1-\zeta) \nabla_z-\omega\nabla_r-r^{-1}\nabla_\omega.
\end{align*}
Then the Fisher information $\widetilde I_{ij}(g)$ can be written as the following sum of parallel, radius and angular Fisher information
\begin{align*}
&\widetilde I_{ij}(g)=\widetilde I_{z}(g)+\widetilde I_{r}(g)+\widetilde I_{\omega}(g)\\
&:=\frac{1}{m_i+m_j}\int_{\R^d\times \R_+\times S^{d-1}}{|\nabla_z \log g|^2}{g}r^{d-1}\dd\omega\dd r\dd z\\
&\quad +\Big(\frac{1}{m_i}+\frac{1}{m_j}\Big)\int_{\R^d\times \R_+\times S^{d-1}}{|\nabla_r  \log  g|^2}{g}r^{d-1}\dd\omega\dd r\dd z\\
&\quad+\Big(\frac{1}{m_i}+\frac{1}{m_j}\Big)\int_{\R^d\times \R_+\times S^{d-1}}{|\nabla_\omega   \log  g|^2}gr^{d-3}\dd\omega\dd r\dd z.
\end{align*}
Then to show \eqref{goal-2}, we only need to show 
\begin{align*}
   \big\langle \widetilde I_{z}'(g),\widetilde Q_{ij}(g,g)\big\rangle+\big\langle \widetilde I_{r}'(g),\widetilde Q_{ij}(g,g)\big\rangle+\big\langle \widetilde I_{\omega}'(g),\widetilde Q_{ij}(g,g)\big\rangle\le 0.
\end{align*}
By applying \cite[Lemma 3.1-3.3 \& 5.1]{ISV26}, we have 
 \begin{align*}
     &\langle \widetilde I_{z}'(g),\widetilde Q_{ij}(g,g)\rangle\\
     =&{}-\frac{1}{2(m_i+m_j)}\int|\nabla_z\log g'-\nabla_z \log g|^2(g+g')b_{ij}(\omega\cdot\omega')\alpha_{ij}(r)r^{d-1}\le0,
    \end{align*}
 and
  \begin{equation}
      \label{goal-3}
   \begin{aligned}
     &\langle \widetilde I_{r}'(g),\widetilde Q_{ij}(g)\rangle+\langle \widetilde I_{\omega}'(g),\widetilde Q_{ij}(g)\rangle\\
     \le&{}\frac12\Big(\frac{1}{m_i}+\frac{1}{m_j}\Big)\int_{\R^d\times\R_+\times(S^{d-1})^2}\frac{|g'-g|^2}{g'+g}b_{ij}(\omega\cdot\omega')\frac{|\alpha_{ij}'|^2}{\alpha_{ij}}r^{d-1}\\
     &-4\alpha_{ij}r^{d-3}g\Gamma^2_{\cB_{ij},\Delta}\big(\log g(z,r,\omega'),\log g(z,r,\omega')\big)\dd\omega'\dd\omega\dd r\dd z.
    \end{aligned}
    \end{equation}
   By 
log-Sobolev inequality \eqref{log-sobolev}, the inequality \eqref{goal-3} holds negatively if $\alpha_{ij}$ and $b_{ij}$ satisfy the assumption \eqref{ass}, which completes the proof of this theorem. 

\end{proof}

\appendix
\section{Grazing limit}
\label{sec:GL}
In this Appendix, we show that the grazing limit of the multispecies Boltzmann equation \eqref{IBE}-\eqref{B} leads to the Landau equation \eqref{ILE}-\eqref{A}. The proof follows the classical results in \cite{villani1998new}. For the reason of completeness, we sketch the proof here.

We recall the definition deviation angle \eqref{B} 
\begin{equation}
\label{angle}
    \theta_{ij}=\arccos\frac{v_i-v_j}{|v_i-v_j|}\cdot \omega.
\end{equation}
We recall the Boltzmann collision kernel \eqref{B}
\begin{align*}
B=(B_{ij})_{ij=1}^N,\quad B_{ij}=\alpha_{ij}(|v_i-v_j|)b_{ij}(\theta_{ij}),
\end{align*}
where $\alpha_{ij},\,b_{ij}\ge 0$ and $\alpha_{ij}=\alpha_{ji}$, $b_{ij}b_{ji}$. Although $f_i'f_j'$ ($i\neq j$) is not symmetric under the transformation $\omega\mapsto -\omega$, we can still restrict the support of $\theta_{ij}$ to $[0,\pi/2]$ by symmetrisation. Indeed, define
 $\theta_{ij}(\theta)=\frac12\big(\theta_{ij}(\theta)+\theta_{ij}(-\theta)\big)$. Since we assume $b_{ij}=b_{ji}$ and combine the 
$ij$- and $ji$-terms in the weak formulation \eqref{weak-B}, this modification does not change the desired form \eqref{sum}. Consequently, we assume  that $\supp(b_{ij})\subset[0,\pi/2]$. 
We define the angle function
\begin{align*}
\beta_{ij}(\theta)=\sin\theta^{d-2}b_{ij}(\theta),\quad \supp(\beta_{ij})\subset[0,\pi/2].
\end{align*}
Moreover, we assume that $\beta_{ij}(\theta)\gtrsim \theta^{-\nu}$ for some  $\nu\in(0,2)$ and the following angular momentum is bounded 
\begin{equation}
\label{mome-A}
 \int_0^{\pi/2}\theta^2\beta_{ij}(\theta)\dd\theta=\frac{2(d-1)}{|S^{d-2}|}\frac{(m_i+m_j)^2}{m_i^2m_j^2}.
\end{equation}
For $d=2$, we take $|S^0|=2$. The constant of angular momentum is chosen for normalisation reason.

For $\varepsilon\in(0,1)$, we take a sequence of scaling
\begin{equation*}
\beta_{ij}^\varepsilon(\theta)=\pi^3/\varepsilon^3\beta_{ij} (\pi\theta/\varepsilon).
\end{equation*}
Notice that the angular momentum assumption \eqref{mome-A} holds for $\beta_{ij}^\varepsilon(\theta)$.
Correspondingly, we define the scaling kernels
\begin{equation}
\label{def:B-varepsilon}
b_{ij}^\varepsilon(\theta)=\sin\theta^{-(d-2)}\beta_{ij}^\varepsilon(\theta)\quad\text{and}\quad B_{ij}^\varepsilon= \alpha_{ij}(|v_i-v_j|)b_{ij}^\varepsilon(\theta_{ij}).
\end{equation}
Let $k_{ij}:=\frac{v_i-v_j}{|v_i-v_j|}$ and $S^{d-2}_{k_{ij}^\perp}:=\{\gamma\in S^{d-1}\mid \gamma\cdot k_{ij}=0\}$. Notice that $\int_{S^{d-1}}b_{ij}(\omega\cdot k_{ij})\dd\omega=\int_0^{\varepsilon/2}\int_{S^{d-2}_{k_{ij}^\perp}}\beta_{ij}(\theta_{ij})\dd\gamma\dd \theta_{ij}$. 

Following the grazing collision limit for Boltzmann equations in \cite{villani1998new,DH25C}, we have the following lemma.
\begin{lemma}
\label{lem:grazing}
  Let $\varepsilon\in(0,1)$.  Let $\tau_i=\tau_i(f)=1+\alpha_i f$ and $\alpha_i\in\{-1,0,1\}$.  We have  
\begin{align*}
&\lim_{\theta_{ij}\to0}\theta_{ij}^{-2}\int_{S^{d-2}_{k_{ij}^\perp}}\tau_i'\tau_j'\overline\nabla^{ij} \Phi =\frac{|S^{d-2}|(m_im_j)^2}{2(d-1)(m_i+m_j)^2}\times\\
&\Big(2|v_i-v_j|^2\big(m_i^{-1}\nabla_{v_i}-m_j^{-1}\nabla_{v_j}\big)\tau_i\tau_j\cdot \Pi_{ij}\big(m_i^{-1}\nabla_{v_i}\phi_i-m_j^{-1}\nabla_{v_j}\phi_j\big)\\
&\quad+\tau_i\tau_j\big(m_i^{-1}\nabla_{v_i}-m_j^{-1}\nabla_{v_j}\big)\cdot\big(|v_i-v_j|^2\Pi_{_{ij}}\big(m_i^{-1}\nabla_{v_i}\phi_i-m_j^{-1}\nabla_{v_j}\phi_j\big)\big)\Big).
\end{align*}
\end{lemma}

Let $Q^{\aB,\varepsilon} (F,F)$ be the collision operator given by \eqref{B:Q-ij} associated to kernel $B^\varepsilon$ \eqref{def:B-varepsilon}. By Lemma \ref{lem:grazing}, we have 
\begin{align*}
    Q^{\aB,\varepsilon}(F,F)\to \QL(F,F).
\end{align*}
Indeed, by weak formulation \eqref{weak-B}, we have 
{\begin{equation}
\label{sum}
\begin{aligned}
&\lim_{\varepsilon\to0}\big\langle Q^{\aB,\varepsilon}(F,F),\Phi\big\rangle\\ =&{}\lim_{\varepsilon\to0}\frac12\sum_{i,j=1}^N\int_{\R^{3d}}\alpha_{ij} f_if_j\int_{0}^{{\varepsilon/2}} \beta^\varepsilon_{ij} \int_{S^{d-2}_{k_{ij}^\perp}}\tau_i'\tau_j'\overline\nabla^{ij} \Phi \\
=&{}\frac12\sum_{i,j=1}^N\frac{|S^{d-2}|(m_im_j)^2}{2(d-1)(m_i+m_j)^2}\int_{0}^{{\varepsilon/2}}\theta_{ij}^2\beta^\varepsilon_{ij}\times\\
&\Big(\int_{\R^{3d}}\alpha_{ij} |v_i-v_j|^2\big(f_if_j (m_i^{-1}\nabla_{v_i}-m_j^{-1}\nabla_{v_j})\tau_i\tau_j\\
&-\tau_i\tau_j (m_i^{-1}\nabla_{v_i}-m_j^{-1}\nabla_{v_j})f_if_j\big)\cdot \Pi_{ij}(m_i^{-1}\nabla_{v_i}\phi_i-m_j^{-1}\nabla_{v_j}\phi_j)\\
=&{}-\frac12\sum_{i,j=1}^N\int_{\R^{3d}}\alpha_{ij}|v_i-v_j|^2 \widetilde\nabla^{ij} \Phi\cdot\Big(f_if_j\tau_i\tau_j \wn^{ij}\dd\cH(F)\Big)\\
=&{}\langle \QL (F,F),\Phi\rangle,
\end{aligned}
\end{equation}}
where the Landau collision kernel is indeed given by \eqref{A}
\begin{align*}
    A_{ij}(|v_i-v_j|)=\frac{|v_i-v_j|^2\alpha_{ij}(|v_i-v_j|)}{m_i}.
\end{align*}
In the above, we use the fact $[\wn^{ij},|v_i-v_j| ]=0$ and the angular momentum assumption \eqref{mome-A}.

We left to show Lemma \ref{lem:grazing}.
\begin{proof}[Proof of Lemma \ref{lem:grazing}] We sketch the proof by following \cite{villani1998new,DH25C} with appropriated modification.

In the following, we fix a pair $i,\,j=1,\dots,N$. For notation convenience, we write $\theta=\theta_{ij}$ and $k=k_{ij}$. We define 
\begin{align*}
   c_{ij}=\frac{m_i}{m_i+m_j}\quad\text{and} \quad  c_{j}=\frac{m_j}{m_i+m_j}.
\end{align*}
By definition, we have 
\begin{align*}
v_i'-v_i=c_{ji}|v_i-v_j|(\omega-k)\quad\text{and}\quad v_j'-v_j=-c_{ij}|v_i-v_j|(\omega-k).   
\end{align*}
By Taylor expansion, $\overline\nabla^{ij}\Phi$ can be written as
\begin{align*}
\overline\nabla^{ij}\Phi&=|v_i-v_j|\underbrace{(\omega-k)\cdot\big(c_{ji}\nabla_{v_i}\phi_i-c_{ij}\nabla_{v_j}\phi_j\big)}_{=:\Theta_1\sim|\theta|}\\
&+\frac{|v_i-v_j|^2}{(m_i+m_j)^2}\underbrace{(\omega-k)\otimes (w-k):T}_{=:\Theta_2\sim\theta^2},\\
\tau_i'\tau_j'&=\tau_i\tau_j+|v_i-v_j|\underbrace{(\omega-k)\cdot (c_{ji}\nabla_{v_i}-c_{ij}\nabla_{v_j})\tau_i\tau_j}_{=:\Theta_3\sim |\theta|}+O(|\omega-k|^2),
\end{align*}
where $T$ is given by 
\begin{align*}
T&:=\int_0^1t\int_0^1\Big(m_j^2D^2_{v_i}\phi_i\big(x,s(tv_i'+(1-t)v)+(1-s)v_i\big)\\
&+m_i^2D^2_{v_j}\phi_j\big(x,s(tv_j'+(1-t)v)+(1-s)v_j\big)\Big)\dd s\dd t.
\end{align*}
The definition of the deviation angle \eqref{angle} implies that 
\begin{equation}
\label{omega-k}
\omega-k= k(\cos \theta-1)+\gamma\sin \theta, 
\end{equation}
where $\gamma=\Pi_{k_{ij}^\perp}(\omega-k) \in S^{d-2}_{ k_{ij}^\perp}$, $(\cos \theta-1)=-\frac{ \theta^2}{2}+o(\theta^2)$ and $\sin \theta= \theta+o( \theta)$.

By using \cite[Proposition 3.2]{DGH25}
\begin{equation}
\label{perp}
\begin{aligned}
\int_{S^{d-2}_{k_{ij}^\perp}} (\omega- k)\otimes(\omega-k) =\frac{|S^{d-2}|}{d-1}\Pi_{k_{ij}^\perp}+o(\theta^2),
\end{aligned}
\end{equation}
we have 
\begin{equation*}
\begin{aligned} &\lim_{\theta\to0}\theta^{-2}\int_{S^{d-2}_{\hat k_{ij}^\perp}}   \Theta_1\Theta_3\\
 =&{}\lim_{\theta\to0}\theta^{-2}\int_{S^{d-2}_{ k_{ij}^\perp}}   (\omega-k)\otimes (\omega- k):\big(c_{ji}\nabla_{v_i}-c_{ij}\nabla_{v_j}\big)\tau_i\tau_j\otimes \big(c_{ji}\nabla_{v_i}\phi_i-c_{ij}\nabla_{v_j}\phi_j\big)\\
  =&{}\frac{|S^{d-2}|\theta^2}{(d-1)}\big(c_{ji}\nabla_{v_i}-c_{ij}\nabla_{v_j}\big)\tau_i\tau_j \cdot\Pi_{ k_{ij}^\perp}\big(c_{ji}\nabla_{v_i}\phi_i-c_{ij}\nabla_{v_j}\phi_j\big).
\end{aligned}
\end{equation*}

Then to show Lemma \ref{lem:grazing}, we only need to show that
\begin{equation}
\label{int-0}
\begin{aligned}
&\lim_{\theta\to0}\theta^{-2}\int_{S^{d-2}_{k_{ij}^\perp}}\overline\nabla^{ij}\Phi \\
=&{}\lim_{\theta\to0}\theta^{-2}\int_{S^{d-2}_{k_{ij}^\perp}}|v_i-v_j| \Theta_1+\frac{|v_i-v_j|^2}{(m_i+m_j)^2}\Theta_2\\
=&{}\frac{|S^{d-2}|}{2(d-1)}(c_{ji}\nabla_{v_i}-c_{ij}\nabla_{v_j})\cdot (|v_i-v_j|^2\Pi_{ij})\cdot \big(c_{ji}\nabla_{v_i}\phi_i-c_{ij}\nabla_{v_j}\phi_j\big).
\end{aligned}    
\end{equation}
Concerning the $\Theta_1$ term, \eqref{omega-k} implies that 
\begin{equation*}
\begin{aligned}
    &\int_{S^{d-2}_{k_{ij}^\perp}}\Theta_1=\int_{S^{d-2}_{k_{ij}^\perp}}(\cos\theta-1)N_k+\sin\theta N_\gamma,
\end{aligned}
\end{equation*}
where $N_k$ and $N_\gamma$ denote the projections 
\begin{align*}
N_k=\big(c_{ji}\nabla_{v_i}\phi_i-c_{ij}\nabla_{v_j}\phi_j\big)\cdot k\quad\text{and}\quad  N_\gamma=\big(c_{ji}\nabla_{v_i}\phi_i-c_{ij}\nabla_{v_j}\phi_j\big)\cdot \gamma. \end{align*}
Notice that $N_k$ is independent of $\gamma\in S^{d-2}_{k_{ij}^\perp}$, and $\int_{S^{d-2}_{k_{ij}^\perp}} N_\gamma\dd \gamma=0$.
Hence, we have
\begin{equation}
\label{int-1}
\begin{aligned}
  &\lim_{\theta\to0}\theta^{-2}|v_i-v_j|\int_{S^{d-2}_{k_{ij}^\perp}}\Theta_1 \\
  =&{}-\frac{|S^{d-2}|}{2}(v_i-v_j)\cdot \big(c_{ji}\nabla_{v_i}\phi_i-c_{ij}\nabla_{v_j}\phi_j\big)\\
  =&{}\frac{|S^{d-2}|}{2(d-1)}\big((c_{ji}\nabla_{v_i}-c_{ij}\nabla_{v_j})\cdot (|v_i-v_j|^2\Pi_{k_{ij}^\perp})\big)\cdot \big(c_{ji}\nabla_{v_i}\phi_i-c_{ij}\nabla_{v_j}\phi_j\big),
\end{aligned}
\end{equation}
where to show the last equality, we use  the following identity 
\begin{align*}
(c_{ji}\nabla_{v_i}-c_{ij}\nabla_{v_j})\cdot(|v_i-v_j|^2\Pi_{k_{ij}^\perp})=-(d-1)(v_i-v_j).  
\end{align*}

Concerning the $\Theta_2$ term,   as $|\theta|\to0$, we have
\begin{align*}
T\to\frac{m_j^2D^2_{v_i}\phi_i+m_i^2D^2_{v_j}\phi^j}{2}.
\end{align*}
By using \eqref{perp}, we have 
\begin{equation}
\label{int-2}
\begin{aligned}
&\lim_{\theta\to0}\theta^{-2}\frac{|v_i-v_j|^2}{(m_i+m_j)^2}\int_{S^{d-2}_{k_{ij}^\perp}}\Theta_2\dd\gamma\\
=&{}\frac{|v_i-v_j|^2}{2(m_i+m_j)^2}\int_{S^{d-2}_{k_{ij}^\perp}}\gamma\otimes \gamma\dd \gamma:\big(m_j^2D^2_{v_i}\phi_i+m_i^2D^2_{v_j}\phi_j\big)\\
=&{}\frac{|S^{d-2}|}{2(d-1)}|v_i-v_j|^2\Pi_{k_{ij}^\perp}:(c_{ji}\nabla_{v_i}-c_{ij}\nabla_{v_j})\big(c_{ji}\nabla_{v_i}\phi_i-c_{ij}\nabla_{v_j}\phi_j\big).
\end{aligned}
\end{equation}

Hence, sum \eqref{int-1} and \eqref{int-2} implies \eqref{int-0}, where we use 
$\Pi_{k_{ij}^\perp}(c_{ji}\nabla_{v_i}-c_{ij}\nabla_{v_j})|v_i-v_j|^2=0$.

\end{proof}
\printbibliography

\end{document}